\documentclass{amsart}

\usepackage{epsf,epsfig,amsfonts,amsgen,amsmath,amstext,amsbsy,amsopn,amsthm,cases,listings,color
}
\usepackage{ebezier,eepic}
\usepackage{color}
\usepackage{multirow}
\usepackage{epstopdf}
\usepackage{graphicx}
\usepackage{pgf,tikz}
\usepackage{mathrsfs}
\usepackage[marginal]{footmisc}
\usepackage{enumitem}
\usepackage{appendix}
\usepackage{graphicx}
\usepackage{ amssymb,amsmath,amscd, enumitem,
	latexsym,amsthm, verbatim}
\usepackage{amssymb,latexsym,eufrak,amsmath,amscd}
\usepackage[all]{xy}
\usepackage{pdfsync}
\usepackage{url}

\usepackage{pgf,tikz,pgfplots}
\usepackage{subfigure}
\usepackage{mathrsfs}
\usetikzlibrary{arrows}

\newtheorem{theorem}{Theorem}[section]
\newtheorem{lemma}[theorem]{Lemma}

\newtheorem{corollary}[theorem]{Corollary}
\newtheorem{proposition}[theorem]{Proposition}

\theoremstyle{definition}
\newtheorem{definition}[theorem]{Definition}
\newtheorem{example}[theorem]{Example}

\theoremstyle{remark}
\newtheorem{remark}[theorem]{Remark}
\newtheorem{claim}[theorem]{Claim}

\numberwithin{equation}{section}

\newcommand{\dH}{{\mathcal{H}}}

\newcommand{\dE}{{\mathcal{E}}}
\newcommand{\dF}{{\mathcal{F}}}
\newcommand{\dG}{{\mathcal{G}}}
\newcommand{\rk}{\textnormal{rk}}

\newcommand{\Hom}{\textnormal{Hom}}

\newcommand{\Symp}{\textnormal{Symp}}
\newcommand{\Image}{\textnormal{Im}}
\newcommand{\Aut}{\textnormal{Aut}}
\newcommand{\coker}{\textnormal{coker}}
\newcommand{\id}{\textnormal{Id}}

\begin{document}

\title{Stability of Special Generalized Null Correlation Bundles on $\mathbb{P}^{5}$}


\author{}
\address{}
\curraddr{}
\email{}
\thanks{}

\author{Shijie Shang}
\address{Department of Mathematics, Statistics, and Computer Science, University of Illinois at Chicago, Chicago, IL 60607}
\curraddr{}
\email{sshang8@uic.edu}
\thanks{}

\subjclass[2010]{14D20, 14J60}


\date{\today}

\dedicatory{}

\begin{abstract}
In this paper, we study special generalized null correlation bundles on $\mathbb{P}^{5}$. We prove that special generalized null correlation bundles on $\mathbb{P}^{5}$ are stable under some numerical conditions. Moreover, we prove that the closure of the subvariety parametrizing stable special generalized null correlation bundles is an irreducible component of the moduli space of rank $4$ stable vector bundles with corresponding Chern classes on $\mathbb{P}^{5}$. As an application, we prove that the number of irreducible components of moduli space of rank $4$ stable vector bundles on $\mathbb{P}^5$ with some fixed Chern classes can be arbitrarily high.
\end{abstract}

\maketitle

\section{Introduction}
A generalized null correlation bundle on $\mathbb{P}^3$ is the cohomology  bundle $\dE$ of the following monad
$$
\mathcal{O}_{\mathbb{P}^3}(-c)\rightarrow\mathcal{O}_{\mathbb{P}^3}(a)\oplus\mathcal{O}_{\mathbb{P}^3}(b)\oplus\mathcal{O}_{\mathbb{P}^3}(-a)\oplus\mathcal{O}_{\mathbb{P}^3}(-b)\rightarrow\mathcal{O}_{\mathbb{P}^3}(c)
$$
where $c>a\geq b\geq0$. If $a=b=0$ and $c=1$, one obtains the null correlation bundle constructed by Barth \cite{Barth}. $\dE$ is stable if and only if $c>a+b$. We will denote by $N(a,b,c)$ the variety parameterizing generalized null correlation bundles of the preceding type in the case $c>a+b$. 

Let $\mathcal{M}_{\mathbb{P}^N}(r;c_1,c_2,\cdots,c_r)$ be the moduli space of stable rank $r$ vector bundles on $\mathbb{P}^N$ with fixed Chern classes $c_1$, $c_2$, $\cdots$, $c_r$. Ein \cite{Ein88} proved that the closure of $N(a,b,c)$ is an irreducible component of $\mathcal{M}_{\mathbb{P}^3}(2;0,c^2-a^2-b^2)$. Moreover, every $p\in N(a,b,c)$ is a smooth point of $\mathcal{M}_{\mathbb{P}^3}(2;0,c^2-a^2-b^2)$. As an application, Ein \cite{Ein88} proved that there exist moduli spaces $\mathcal{M}_{\mathbb{P}^3}(2;0,t)$ with an arbitrarily high number of irreducible components. Ballico and Miró-Roig \cite{BM} generalized the result to projective threefold by using a finite map from a smooth projective threefold to $\mathbb{P}^3$. They proved that the number of irreducible components of the moduli space of rank 2 $H$-stable bundles $\dE$ on $X$ with fixed $c_1(\dE)=c_1$ and $c_2\cdot H=d$ tends to infinity as $d$ goes to infinity when $ac_1$ is numerically equivalent to the polarization $bH$ for some integers $a$ and $b$.  Ancona and Ottaviani \cite{AO95} proved that there exist moduli spaces $\mathcal{M}_{\mathbb{P}^5}(3;0,t,0)$ with an arbitrarily high number of irreducible components.

In this paper, we generalize the definition of the generalized null correlation bundle on $\mathbb{P}^{3}$ to $\mathbb{P}^{2n+1}$. A generalized null correlation bundle $\dE$ on $\mathbb{P}^{2n+1}$ is the cohomology bundle of the following monad
$$
\mathcal{O}_{\mathbb{P}^{2n+1}}(-c)\xrightarrow{\beta}(\oplus^{n+1}_{i=1}\mathcal{O}_{\mathbb{P}^{2n+1}}(a_{i}))\oplus(\oplus^{n+1}_{i=1}\mathcal{O}_{\mathbb{P}^{2n+1}}(-a_{i}))\xrightarrow{\alpha}\mathcal{O}_{\mathbb{P}^{2n+1}}(c)
$$
where $c>a_{1}\geq\cdots\geq a_{n+1}$ are all nonnegative integers. We will specify the forms of $\alpha$ and $\beta$ such that the vector bundle $\dE$ is symplectic and call this kind of bundle special generalized null correlation bundle (See Definition \ref{def2.1}). First, we prove the stability of special generalized null correlation bundles on $\mathbb{P}^5$ under some numerical conditions.
\begin{theorem}\textnormal{(See Theorem \ref{thm3.5}).}
Let $\dE$ be a special generalized null correlation bundle on $\mathbb{P}^{5}$ with $a_1\geq a_2\geq a_3\geq 0$ and $c>2a_1+a_2$. Then $\dE$ is stable.
\end{theorem}

Let $t_i:=c_i(\dE)$, for $i=1,2,3,4$, and we denote by $N(c,a_1,a_2,a_3)$ the subvariety of $\mathcal{M}_{\mathbb{P}^{5}}(4;t_1,t_2,t_3,t_{4})$ parametrizing special generalized null correlation bundles on $\mathbb{P}^{5}$. By the classification theorem of self-dual monads with symplectic structure due to Barth and Hulek \cite{BH78}, we compute the dimension of $N(c,a_1,a_2,a_3)$. Meanwhile, we show that, for any $p\in N(c,a_1,a_2,a_{3})$, $h^1(\mathbb{P}^{5},\dE nd\,\dE_p)$ is exactly equal to $\dim N(c,a_1,a_2,a_{3})$, where $\dE_p$ is the special generalized null correlation bundle corresponding to $p$. Hence we can deduce the following theorem.
\begin{theorem}\label{Thm1.2}\textnormal{(See Theorem \ref{thm4.9})}.
For $a_1\geq a_2\geq a_3\geq 0$ and $c>5a_1$,\\
(i) the closure of $N(c,a_1,a_2,a_{3})$ in $\mathcal{M}_{\mathbb{P}^{5}}(4;t_1,t_2,t_3,t_{4})$ is an irreducible component of $\mathcal{M}_{\mathbb{P}^{4}}(4;t_1,t_2,t_3,t_{4})$;\\
(ii) if $p\in N(c,a_1,a_2,a_{3})$, then $p$ is a smooth point of $\mathcal{M}_{\mathbb{P}^{5}}(4;t_1,t_2,t_3,t_{4})$.
\end{theorem}
\begin{remark}
Ein \cite{Ein88} proved similar results for $\mathbb{P}^3$, but his method depends on Rao's theorem \cite{Rao} and the Serre construction between rank 2 vector bundles and codimension 2 subvarieties, which may not be generalized to higher rank cases. 
\end{remark}

 As an application of Theorem \ref{Thm1.2}, we prove the following proposition.
\begin{proposition}\textnormal{(See Proposition \ref{prop6.8})}.
Suppose that $\mathcal{M}_{\mathbb{P}^5}(4;0,s,0,t)=X_1\cup X_2\cdots\cup X_{m_{s,t}}$ where $X_i$'s are distinct irreducible components of $\mathcal{M}_{\mathbb{P}^5}(4;0,s,0,t)$. Then there exists a sequence of pairs $\{(s_i,t_i)\in\mathbb{Z}^2|i=1,2,\cdots\}$ such that 
$$\lim_{i\rightarrow+\infty} m_{s_i,t_i}=+\infty.
$$ 
\end{proposition}

\textbf{Organization of the paper.}
In section 2, we review some preliminaries that  will be used in the proofs. In section 3, we prove some properties of special generalized null correlation bundles. In section 4, we prove that the dimension of $N(c,a_1,a_2,a_{3})$ is equal to the dimension of the Zariski tangent space of $\mathcal{M}_{\mathbb{P}^{5}}(4;t_1,t_2,t_3,t_{4})$ at $p$ for $p\in N(c,a_1,a_2,a_{3})$. In section 5, we show that the number of irreducible components of moduli spaces of rank 4 stable vector bundles with some fixed Chern class on $\mathbb{P}^5$ can be arbitrarily high. 

\textbf{Notations and Conventions.}
Throughout this paper, we work over the gound field $\mathbb{C}$. We do not distinguish between vector bundles and locally free sheaves. For any sheaf $\dE$ on a projective variety $X$, we will let $H^i_*(X,\dE)=\oplus_{v\in\mathbb{Z}}H^i(X,\dE(v))$ for any $i\in\mathbb{Z}_{\geq0}$. We will write $H^i(X,\dE)$ as $H^i(\dE)$ and $h^i(X,\dE)$ as $h^i(\dE)$ when no confusion can arise. We will denote by $\mathcal{M}_{\mathbb{P}^N}(r;c_1,\cdots,c_r)$ the moduli space of stable rank $r$ vector bundles on $\mathbb{P}^N$ with fixed Chern classes $c_i$, for $i=1,\cdots,r$. For simplicity, we will write $\mathcal{O}_{\mathbb{P}^N}(a)$ as $\mathcal{O}(a)$ for any projective space $\mathbb{P}^N$.

\section{Preliminaries}
\subsection{}
For basic facts and notations about vector bundles we refer to \cite{OSS}. We use the definition of stability of Mumford-Takemoto. 
\begin{definition}\label{def1.9}
	(i) A vector bundle $\mathcal{F}$ is called \textit{symplectic} if there exists an isomorphism $f:\dF\rightarrow \dF^*$ such that $f^T=-f$. This is equivalent to the existence of a nondegenerate form $\Gamma\in H^0(\wedge^2 \dF)$. We will write the symplectic bundle $\dF$ as $(\dF,f)$ if we want to emphasize its symplectic structure. \\
	(ii) If $\varphi:(\dF_1,f_1)\rightarrow(\dF_2,f_2)$ is an isomorphism of vector bundles and $\varphi^Tf_2\varphi=f_1$, then $\varphi$ is called a \textit{symplectic isomorphism}. For a symplectic vector bundle $(\dF,f)$, we will denote by $\Symp(\dF,f)$ the group of symplectic automorphisms of $(\dF,f)$. 
\end{definition}

\begin{example}\label{eg1.10} Given $m$ nonnegative integers $a_1,a_2,\cdots,a_m$,
	$$
	\dH:=\mathcal{O}(a_1)\oplus\mathcal{O}(a_2)\oplus\cdots\oplus\mathcal{O}(a_{m})\oplus\mathcal{O}(-a_{1})\oplus\mathcal{O}(-a_2)\cdots\oplus\mathcal{O}(-a_m)
	$$
	is a vector bundle on a projective manifold $\mathbb{P}^N$. Then 
	$$
	\dH^*=\mathcal{O}(-a_1)\oplus\mathcal{O}(-a_2)\oplus\cdots\oplus\mathcal{O}(-a_{m})\oplus\mathcal{O}(a_1)\oplus\mathcal{O}(a_2)\oplus\cdots\oplus\mathcal{O}(a_{m})
	$$
	Set $J_{2m}:=\begin{pmatrix}
	0 & I_m\\
	-I_m & 0
	\end{pmatrix}$, where $I_m$ is the $m\times m$ identity matrix. Observe that $J_{2m}\in \Hom(\dH,\dH^*)$ and $J_{2m}^T=-J_{2m}$. Then $(\dH,J_{2m})$ is a symplectic vector bundle. 
\end{example}
\begin{lemma} \label{lem2.11}
	Let $\dH=(\oplus^{m}_{i=1}\mathcal{O}(a_i))\oplus(\oplus^{m}_{i=1}\mathcal{O}(-a_i))$ be a vector bundle on $\mathbb{P}^N$, where $a_1\geq a_2\geq\cdots\geq a_{m}$ are all nonnegative integers. Then $\dim\Symp(\dH,J_{2m})=h^0(\dH\otimes\dH)-h^0(\wedge^2\dH)$.
\end{lemma}

\begin{proof} 
	Since any element in $\Hom(\dH,\dH)$ can be written as a $2m\times2m$ matrix whose entries are homogeneous polynomials, then $\Hom(\dH,\dH)$ can be considered as a complex manifold of dimension $h^0(\dH\otimes\dH)$. 
	\begin{claim}\label{claim1}
		For any $\varphi\in \Hom(\dH,\dH)$, $\det(\varphi)\in\mathbb{C}$.
	\end{claim}
	Suppose $\det(\varphi)$ is a non-constant polynomial. Then $\det(\varphi^n)=(\det(\varphi))^n$ is a non-constant polynomial of degree $n\deg(\det(\varphi))$. However, the degree of the determinant of any element in $\Hom(\dH,\dH)$ is at most $2ma_1$. When $n$ tends to infinity, $\deg(\det(\varphi^n))$ also tends to infinity, which leads to a contradiction. Thus, the claim follows.
	
	Let $\Aut(\dH)$ be the the group of the automorphisms of $\dH$. Then $\Aut(\dH)\subseteq\Hom(\dH,\dH)$. Moreover, we have
	\begin{claim}\label{claim2}
		$\Aut(\dH)=\{\det(\varphi)\in\mathbb{C}^*|\varphi\in\Hom(\dH,\dH)\}$.
	\end{claim}  	
	For any $\varphi\in \Aut(\dH)$, then there exist $\psi\in\Aut(\dH)$ such that $\varphi\psi=\id$. Thus, $\det(\varphi\psi)=\det(\varphi)\det(\psi)=1$. Since both $\det(\varphi)$ and $\det(\psi)$ are polynomials, then $\det(\varphi)\in\mathbb{C}^*$. On the other hand, for $\varphi\in \Hom(\dH,\dH)$ satisfying that $\det(\varphi)\in\mathbb{C}^*$, let $\psi\in\Hom(\dH,\dH)$ be $\textnormal{adj}(\varphi)/\det(\varphi)$, where $\textnormal{adj}(\varphi)$ is the adjoint matrix of $\varphi$. Then $\varphi\psi=\psi\varphi=\id$. So the claim follows.
	
	By Claim \ref{claim1} and \ref{claim2}, $\Aut(\dH)$ is an open complex submanifold of $\Hom(\dH,\dH)$. Hence $\Aut(\dH)$ is of dimension $h^0(\dH\otimes\dH)$. 
	
	Now consider the following holomorphic map
	$$
	F:\Aut(\dH)\rightarrow \mathcal{S}:=\{\phi^T=-\phi|\phi\in\Hom(\dH,\dH^*)\}
	$$
	$$
	\varphi\mapsto \varphi^TJ_{2m}\varphi-J_{2m}
	$$
	Then $dF_\varphi(\psi)=\varphi^TJ_{2m}\psi+\psi^TJ_{2m}\varphi=\varphi^TJ_{2m}\psi-(\varphi^TJ_{2m}\psi)^T$. For any $\phi_0\in\Hom(\dH,\dH^*)$, we have $\varphi^TJ_{2m}(J_{2m}^{T}(\varphi^{T})^{-1}\phi_0)=\phi_0$. Thus, $F$ is a submersion. Since $\Symp(\dH,J_{2m})=F^{-1}(0)$ is the preimage of the regular value 0, $\Symp(\dH,J_{2m})$ is a complex submanifold. Observe that $\mathcal{S}$ is a complex submanifold of dimension $h^0(\wedge^2\dH)$ in $\Hom(\dH,\dH^*)$. Then $\dim\Symp(\dH,J_{2m})=h^0(\dH\otimes\dH)-h^0(\wedge^2\dH)$.
\end{proof}
\begin{remark}
When $a_1=a_2=\cdots=a_m=0$, the lemma reduces to the dimension formula of $\textnormal{Sp}(2m,\mathbb{C})$: 
$$
\dim\Symp(\dH,J_{2m})=\dim\textnormal{Sp}(2m,\mathbb{C})=m(2m+1).
$$
\end{remark}
\subsection{}
\begin{definition}(\cite{BH78})
	A \textit{monad} on a smooth projective variety $X$ of dimension $n$ is a sequence of vector bundles
	\begin{equation}\label{eqn3}
	\mathcal{A}\xrightarrow{a}\mathcal{B}\xrightarrow{c}\mathcal{C}
	\end{equation}
	with $ca=0$, $a$ injective with $\Image(a)\subset\mathcal{B}$ a subbundle, and $c$ surjective.
	
	The diagram
	$$
	\xymatrix{
		{} & 0\ar[d]                         & 0\ar[d]                         & {} &  {}  \\
		{} & \mathcal{A}\ar[d]^{}\ar@{=}[r]  & \mathcal{A}\ar[d]^{a} & {} &  {}  \\
		0\ar[r] & \mathcal{K}\ar[r]^{}\ar[d] & \mathcal{B}\ar[r]^{c}\ar[d]^{} & \mathcal{C}\ar[r]\ar@{=}[d] & 0\\
		0\ar[r] & \mathcal{F}\ar[r]\ar[d]           & \mathcal{P}\ar[r]^{}\ar[d] & \mathcal{C}\ar[r]       & 0\\
		{} & 0 & 0 &  & 
	}
	$$
	is called the \textit{display of the monad} (\ref{eqn3}). Here $\mathcal{K}:=\ker c$, $\mathcal{P}:=\coker\;a$, and $\dF:=\ker c/\Image\;a$ is called the \textit{cohomology bundle of the monad} (\ref{eqn3}).
\end{definition}
\begin{remark}\label{rmk2.21}
	$\rk\,\dF=\rk\,\mathcal{B}-\rk\,\mathcal{A}-\rk\,\mathcal{C}$, $c(\mathcal{P})=c(\mathcal{B})c(\mathcal{A})^{-1}$ and $c(\dF)=c(\mathcal{B})c(\mathcal{A})^{-1}c(\mathcal{C})^{-1}$.
\end{remark}
Now we apply the monad construction to symplectic bundle $(\dF,f)$.
\begin{definition} (\cite{BH78})
	A monad
	\begin{equation}\label{eqn1}
	\mathcal{A}\xrightarrow{a}\mathcal{B}\xrightarrow{c}\mathcal{A}^*
	\end{equation}	
	will be called \textit{self-dual} with respect to $(\dF,f)$ if $\dF$ is the cohomology bundle of the monad (\ref{eqn1}) and there is an isomorphism $b:\mathcal{B}\rightarrow\mathcal{B}^*$ with $b=-b^T$, $c=a^Tb$ such that $f$ is induced by $b$ via
	$$
	\xymatrix{
		{\mathcal{A}}\ar[r]^{a}\ar[d]^{- \id} & \mathcal{B}\ar[r]^{c}\ar[d]^{b}                         & \mathcal{A}^*\ar[d]^{\id}                         \\
		\mathcal{A}\ar[r]^{c^T} & \mathcal{B}^*\ar[r]^{a^T}  & \mathcal{A}^*  
	}
	$$
\end{definition}

\begin{theorem}\label{thm1.22}(\cite{BH78}) Let $\varphi:\mathcal{F}_1\rightarrow\dF_2$ be a symplectic isomorphism between two symplectic bundles $(\dF_1,f_1)$ and $(\dF_2,f_2)$, i.e., $\varphi^Tf_2\varphi=f_1$. If $\mathcal{A}_i\xrightarrow{a_i}\mathcal{B}_i\xrightarrow{a^T_ib_i}\mathcal{A}^*_i$ are self-dual monads for $(\mathcal{F}_i,f_i)$ satisfying $\Hom(\mathcal{B}_i,\mathcal{A}_i)=0$, for $i=1,2$, then $\varphi$ is induced by a unique isomorphism of monads:
	$$
	\xymatrix{
		{\mathcal{A}}_1\ar[r]^{a_1}\ar[d]^{\alpha} & \mathcal{B}_1\ar[r]^{a^T_1b_1}\ar[d]^{\beta}                         & \mathcal{A}^*_1\ar[d]^{(\alpha^T)^{-1}}                         \\
		\mathcal{A}_2\ar[r]^{a_2} & \mathcal{B}_2\ar[r]^{a^T_2b_2}  & \mathcal{A}^*_2  
	}
	$$
	Here $\beta$ is symplectic, i.e., $\beta^Tb_2\beta=b_1$.  	
\end{theorem}

\section{Properties of special generalized null correlation bundles}
\begin{definition}\label{def2.1}
For $n\geq 1$, a rank $2n$ bundle $\mathcal{E}$ on $\mathbb{P}^{2n+1}$ is called \textit{generalized null correlation bundle} if it is the cohomology bundle of the following monad,
\begin{equation}\label{eqn13}
	\mathcal{O}(-c)\xrightarrow{\beta}\mathcal{H}\xrightarrow{\alpha}\mathcal{O}(c)
\end{equation}
where $\mathcal{H}=(\mathcal{O}(a_1)\oplus\cdots\mathcal{O}(a_{n+1}))\oplus(\mathcal{O}(-a_{1})\oplus\cdots\oplus\mathcal{O}(-a_{n+1}))$, $c>a_1\geq a_2\geq\cdots\geq a_{n+1}$ are all nonnegative integers, $\alpha$ is a surjection and $\beta\in \textnormal{Hom}(\mathcal{O}(-c),\mathcal{H})$ is a general element. Moreover, the generalized null correlation bundle $\dE$ is said to be \textit{special} if in (\ref{eqn13}),
$$
\beta=(g_1,\cdots,g_{n+1},-f_{1},\cdots,-f_{n+1})^T
$$  
and
$$
\alpha=(f_1,\cdots,f_{n+1},g_{1},\cdots,g_{n+1})
$$
where $f_i\in H^0(\mathcal{O}(c-a_i))$ and $g_{i}\in H^0(\mathcal{O}(c+a_i))$ for $i=1,\cdots,n+1$.	
\end{definition}
We have the display of the monad (\ref{eqn13})
$$
\xymatrix{
	{} & 0\ar[d]                         & 0\ar[d]                         & {} &  {}  \\
	{} & \mathcal{O}(-c)\ar[d]^{}\ar@{=}[r]  & \mathcal{O}(-c)\ar[d]^{\beta} & {} &  {}  \\
	0\ar[r] & \mathcal{F}\ar[r]^{}\ar[d] & \mathcal{H}\ar[r]^{\alpha}\ar[d]^{} & \mathcal{O}(c)\ar[r]\ar@{=}[d] & 0\\
	0\ar[r] & \mathcal{E}\ar[r]\ar[d]           & \mathcal{G}\ar[r]^{}\ar[d] & \mathcal{O}(c)\ar[r]       & 0\\
	{} & 0 & 0 &  & 
}
$$
In the display, $\mathcal{F}=\ker(\alpha)$ and $\dG=\textnormal{coker}(\beta)$ are both rank $2n+1$ bundles.\\
For future reference, we list four short exact sequences from the display.
\[
0\rightarrow\mathcal{O}(-c)\rightarrow\dF\rightarrow\dE\rightarrow 0 \tag{3.a}\label{seqA}
\]
\[
0\rightarrow\mathcal{O}(-c)\rightarrow\dH\rightarrow\dG\rightarrow 0 
\tag{3.b}\label{seqB}
\]
\[
0\rightarrow\dF\rightarrow\dH\rightarrow\mathcal{O}(c)\rightarrow 0
\tag{3.c}\label{seqC}
\]
\[
0\rightarrow\dE\rightarrow\dG\rightarrow\mathcal{O}(c)\rightarrow 0
\tag{3.d}\label{seqD}
\]
Let $S=\mathbb{C}[x_0,x_1,\cdots,x_{2n+1}]$. Set
$$
M=S/<f_1,\cdots,f_{n+1},g_1,\cdots,g_{n+1}>
$$
Observe that $M$ is a 0-dimensional graded complete intersection ring. Denote by $M_j$ the degree $j$ component of $M$. Then $\dim M_0=1$ and $\dim M_j=0$ for $j<0$. 

In the remainder of this section, we assume that $\dH$ is as defined in Definition \ref{def2.1} and $\mathcal{E}$ is a special generalized null correlation bundle on $\mathbb{P}^{2n+1}$, where $n\geq 2$. $\mathcal{F}$ and $\dG$ are as defined in the display of the monad (\ref{eqn13}).
\begin{proposition}\label{prop2.5}
	(i) $H^1_*(\dE)\cong H^1_*(\dF)\cong M\otimes S(c)$.\\
	(ii) $H^i_*(\mathcal{\dE})=0$ for $2\leq i\leq 2n-1$.	
\end{proposition}
\begin{proof}
	(i) For any $t\in\mathbb{Z}$, tensoring \eqref{seqA} by $\mathcal{O}(t)$, we have the following exact sequence
    \begin{equation}\label{3.3.1}
	0\rightarrow\mathcal{O}(t-c)\rightarrow\mathcal{F}(t)\rightarrow\mathcal{E}(t)\rightarrow0.
	\end{equation}
	Passing to cohomology, the exact sequence (\ref{3.3.1}) yields $H^1_*(\mathcal{E})\cong H^1_*(\mathcal{F})$ since $H^1(\mathcal{O}(t-c))\cong H^2(\mathcal{O}(t-c))=0$ for any $t\in\mathbb{Z}$. On the other hand, consider the following exact sequence,
	\begin{equation}\label{3.3.2}
	0\rightarrow\mathcal{F}(t)\rightarrow\mathcal{H}(t)\rightarrow\mathcal{O}(c+t)\rightarrow0
	\end{equation}
	Passing to cohomology, the exact sequence (\ref{3.3.2}) yields $H^1_*(\mathcal{F})\cong M\otimes S(c)$.\\
	(ii) For any $t\in\mathbb{Z}$, tensoring \eqref{seqC} by $\mathcal{O}(t)$, we have the following exact sequence
	\begin{equation}\label{3.3.3}
	0\rightarrow\dF(t)\rightarrow\dH(t)\rightarrow\mathcal{O}(c+t)\rightarrow0
	\end{equation}
	Passing to cohomology, the exact sequence (\ref{3.3.3}) yields $H^i(\dF(t))\cong H^i(\dH(t))=0$ for $2\leq i\leq2n$ since
	$H^i(\mathcal{O}(c+t))=0$ for $1\leq i\leq2n$. On the other hand, consider the following exact sequence
	\begin{equation}\label{3.3.4}
	0\rightarrow\mathcal{O}(t-c)\rightarrow\dF(t)\rightarrow\dE(t)\rightarrow0
	\end{equation}
	Passing to cohomology, the exact sequence (\ref{3.3.4}) yields $H^i(\dE(t))\cong H^i(\dF(t))=0$ for $1\leq i\leq2n-1$
	since $H^i(\mathcal{O}(t-c))=0$ for $1\leq i\leq2n$. Thus, $H^i_*(\mathcal{\dE})=0$ for $2\leq i\leq 2n-1$.
\end{proof}
Similar to Proposition 1.1 in \cite{Ein88}, we have
\begin{proposition}\label{prop2.6}
	Let $F:=H^0_*(\mathcal{F})$ and $H:=(\oplus^{n+1}_{j=1}S(a_j))\oplus(\oplus^{n+1}_{j=1}S(-a_j))$. Then $H^1_*(\mathcal{E})$ and $F$ have the following minimal graded free resolutions over $S$.
	\begin{equation}\label{LES1}
	0\rightarrow \wedge^{2n+2}H\otimes S(-(2n+1)c)\rightarrow \wedge^{2n+1}H\otimes S(-2nc)\rightarrow\cdots
	\end{equation}
	\begin{equation*}
	\rightarrow\wedge^{3}H\otimes S(-2c)\rightarrow\wedge^{2}H\otimes S(-c)\rightarrow H\xrightarrow{\alpha} S(c)\rightarrow H^1_*(\mathcal{E})\rightarrow 0
	\end{equation*}
	\begin{equation}\label{LES2}
	0\rightarrow \wedge^{2n+2}H\otimes S(-(2n+1)c)\rightarrow \wedge^{2n+1}H\otimes S(-2nc)\rightarrow\cdots
	\end{equation}
	\begin{equation*}
	\rightarrow\wedge^{3}H\otimes S(-2c)\rightarrow\wedge^{2}H\otimes S(-c)\rightarrow F \rightarrow 0
	\end{equation*}
\end{proposition}

\begin{proof}
The proposition follows from the Koszul complex associated to the map $\alpha:H\rightarrow S(c)$ immediately.
\end{proof}
\begin{corollary}\label{cor2.7} If $c>2a_1+a_2$, then\\
(i) $h^0(\dF\otimes\dH)=0$.\\
(ii) $h^1(\dE\otimes\dH)=h^0(\dH(c))-h^0(\dH\otimes\dH)$.\\
(iii) $h^0(\dE(c))=h^0(\dF(c))-1=h^0(\wedge^2\dH)-1$.\\
(iv) $h^0(\dE)=h^0(\dF)=0$.\\
\end{corollary}
\begin{proof}
	For (iii), the first equality follows from the exact sequence
$$
0\rightarrow H^0(\mathcal{O})\rightarrow H^0(\mathcal{F}(c))\rightarrow H^0(\mathcal{E}(c))\rightarrow H^1(\mathcal{O})=0.
$$
For (iv), the first equality follows from the exact sequence
$$
0\rightarrow H^0(\mathcal{O}(-c))\rightarrow H^0(\mathcal{F})\rightarrow H^0(\mathcal{E})\rightarrow H^1(\mathcal{O}(-c))=0,
$$
where $H^0(\mathcal{O}(-c))=0$.

The other equalities follow from computing dimensions of some graded pieces of each term in minimal graded free resolutions (\ref{LES1}) and (\ref{LES2}).
\end{proof} 

\begin{proposition}\label{PropR1}
	(i) $\mathcal{F}^*\cong\mathcal{G}$.\\
	(ii) $\mathcal{E}$ is a symplectic bundle.
\end{proposition}
\begin{proof}
	(i)\&(ii) Set $J_{2n+2}:=\begin{pmatrix}
0 & I_{n+1}\\
-I_{n+1} & 0
\end{pmatrix},$ where $I_{n+1}$ is the $(n+1)\times (n+1)$ identity matrix.	
The following commutative diagram
	$$
	\xymatrix{
		{0}\ar[r] & \mathcal{O}(-c)\ar[r]^{\beta}\ar[d]^{-\id}  & \dH\ar[d]^{J_{2n+2}}\ar[r]^{\alpha}                         & \mathcal{O}(c)
	\ar[r]\ar[d]^{\id} &  {0}  \\
		{0}\ar[r] & \mathcal{O}(-c)\ar[r]^{\alpha^T}  & \mathcal{H}^*\ar[r]^{\beta^T} & \mathcal{O}(c)\ar[r] &  {0} 
	}
	$$	
gives an isomorphism between these two monads. Then the proposition follows immediately.
\end{proof}
\begin{proposition}\label{prop2.12}
	(i) If $c>2a_1+a_2$, $\mathcal{E}$ is simple.\\
	(ii) If $c>(2n+1)a_1$, $\mathcal{G}$ and $\mathcal{F}$ are both stable.
\end{proposition}
\begin{proof}	
	(i) Tensoring \eqref{seqA} by $\dE$, we have the following exact sequence
	\begin{equation}\label{3.8.1}
	0\rightarrow\mathcal{E}(-c)\rightarrow\dE\otimes\dF\rightarrow\dE\otimes\dE\rightarrow0
	\end{equation}
	Passing to cohomology, the exact sequence (\ref{3.8.1}) yields the following exact sequence
	$$
	0\rightarrow H^0(\dE(-c))\rightarrow H^0(\dE\otimes\dF)\rightarrow H^0(\dE\otimes\dE)\rightarrow H^1(\dE(-c))
	$$
	Since $\id_\dE\in \Hom(\dE,\dE)$ and $\dE$ is self-dual, then $h^0(\dE^*\otimes\dE)=h^0(\dE\otimes\dE)\geq 1$. Since $H^1_*(\dE)=M\otimes S(c)$, then $h^1(\dE(-c))=1$. To show $h^0(\dE\otimes\dE)=1$, it is enough to show $h^0(\dF\otimes\dE)=0$.
	
	Tensoring \eqref{seqC} by $\dE$, we have the following exact sequence
	\begin{equation}\label{3.8.2}
	0\rightarrow\dF\otimes\dE\rightarrow\dH\otimes\dE\rightarrow\dE(c)\rightarrow0
	\end{equation}
	Passing to cohomology, the exact sequence (\ref{3.8.2}) yields the following exact sequence 
	$$
	0\rightarrow H^0(\dF\otimes\dE)\rightarrow H^0(\dH\otimes\dE)
	$$
	Hence it is enough to show $h^0(\dH\otimes \dE)=0$.
	
	Tensoring \eqref{seqA} by $\dH$, we have the following the exact sequence
	\begin{equation}\label{3.8.3}
	0\rightarrow\dH(-c)\rightarrow\dF\otimes\dH\rightarrow\dE\otimes\dH\rightarrow0
	\end{equation}
	Passing to cohomology, the exact sequence (\ref{3.8.3}) yields the following exact sequence
	$$
	0\rightarrow H^0(\mathcal{H}(-c))\rightarrow H^0(\mathcal{F}\otimes\dH)\rightarrow H^0(\dE\otimes\dH)\rightarrow H^1(\dH(-c))
	$$
	where $H^0(\mathcal{H}(-c))=H^1(\mathcal{H}(-c))=0$ since $c>a_1\geq\cdots\geq a_{n+1}\geq 0$. Thus, $H^0(\mathcal{F}\otimes\dH)\cong H^0(\dE\otimes\dH)$. By Corollary \ref{cor2.7} (i), $h^0(\mathcal{F}\otimes\dH)=0$ since $c>2a_1+a_2$. Hence $h^0(\mathcal{F}\otimes\dE)=h^0(\dE\otimes \dH)=h^0(\mathcal{F}\otimes\dH)=0$. Therefore, $\mathcal{E}$ is simple.
	
	(ii) This follows from Theorem 2.7 in \cite{BS92} immediately.
\end{proof}
\begin{theorem}\label{thm3.5}
	Let $\dE$ be a special generalized null correlation bundle on a smooth complete intersection $X$ of dimension $5$ satisfying $c>2a_1+a_2$. Then $\dE$ is stable.
\end{theorem}
\begin{proof}
	Since $c>2a_1+a_2$, $\dE$ is simple by Proposition \ref{prop2.12} (i). By Proposition \ref{PropR1} (ii), $\dE$ is symplectic. In particular, $\dE\cong \dE^*$. Hence $h^0(\dE\otimes\dE)=h^0(\dE^*\otimes\dE)=1$. By Corollary \ref{cor2.7} (iv), $H^0(\dE)=0$. Therefore, the theorem follows from Theorem 3.5 in \cite{AO94}.
\end{proof}

\section{Moduli space of special generalized null correlation bundles on $\mathbb{P}^5$}
In this section, we will work on $\mathbb{P}^{5}$. We assume that
$\dH$ is as defined in Definition \ref{def2.1} and $\dE$ is a special generalized null correlation bundle on $\mathbb{P}^{5}$ with $a_1\geq a_2\geq a_3\geq 0$ and $c>5a_1$. $\mathcal{F}$ and $\dG$ are as defined in the display of the monad (\ref{eqn13}). Then $\dE$, $\dF$ and $\dG$ are all stable.

\begin{proposition}\label{thm4.1}

	(i) $h^2(\dE nd\;\dE)=h^0(\wedge^2\dH)-1-\chi(\dE(c))$.	\\
	(ii) $h^1(\dE nd\;\dE)=h^0(\wedge^2\dH)-1+h^0(\mathcal{H}(c))-h^0(\mathcal{H}\otimes \dH)$.
\end{proposition}
\begin{proof}
	(i) From the display of the monad (\ref{eqn13}), we have the following short exact sequence
	$$
	0\rightarrow\mathcal{F}\rightarrow\dH\xrightarrow{\alpha}\mathcal{O}(c)\rightarrow0
	$$
	where $\alpha=(f_1,f_2,f_{3},g_{1},g_2,g_{3})$. Tensoring it by $\dF$, we have the following short exact sequence
	\begin{equation*}\label{5.1.1}
	0\rightarrow\dF\otimes\dF\rightarrow\dF\otimes\dH\xrightarrow{1\otimes\alpha}\dF\otimes\mathcal{O}(c)\rightarrow0
	\end{equation*}
	Passing to cohomology, we have the following exact sequence
	$$
	H^1(\dF\otimes\dH)\xrightarrow{H^1(1\otimes\alpha)} H^1(\dF(c))\rightarrow H^2(\dF\otimes\dF)\rightarrow H^2(\dF\otimes\dH)
	$$
	where $H^2(\dF\otimes\dH)=0$. Since $H^1_*(\dF)\cong M\otimes S(c)$, then $H^1(1\otimes\alpha)=0$. Therefore, $H^1(\mathcal{F}(c))\cong H^2(\mathcal{F}\otimes\mathcal{F})$.
	\begin{claim}\label{claim5.2}
		$H^2(\mathcal{F}\otimes\dF)\cong H^2(\dE\otimes\dE)$.
	\end{claim}
Tensoring \eqref{seqA} by $\dE$, we have the following short exact sequence
\begin{equation}\label{5.1.2}
0\rightarrow \dE(-c)\rightarrow\dE\otimes\dF\rightarrow\dE\otimes\dE\rightarrow0
\end{equation}
Passing to cohomology, the exact sequence (\ref{5.1.2}) yields $H^2(\dE\otimes\dE)\cong H^2(\dE\otimes\dF)$ since $H^2(\dE(-c))\cong H^3(\dE(-c))=0$. Tensoring \eqref{seqA} by $\dF$, we have the following short exact sequence
\begin{equation}\label{5.1.3}
0\rightarrow\dF(-c)\rightarrow\dF\otimes\dF\rightarrow\dE\otimes\dF\rightarrow0
\end{equation}
Passing to cohomology, the exact sequence (\ref{5.1.3}) yields $H^2(\dF\otimes\dF)\cong H^2(\dE\otimes\dF)$ since $H^2(\dF(-c))\cong H^3(\dF(-c))=0$. Hence Claim \ref{claim5.2} follows.

Since $\dE$ is self-dual, then $h^2(\dE nd\;\dE)=h^2(\dE\otimes\dE)=h^2(\dF\otimes\dF)=h^1(\dF(c))$. By Proposition \ref{prop2.5}(i), $h^2(\dE nd\;\dE)=h^1(\dF(c))=h^1(\dE(c))$. On the other hand, by Proposition \ref{prop2.5} (ii), $h^i(\dE(c))=0$ for $i=2,3$. By Serre duality, $h^{4}(\dE(c))=h^1(\dE(-c-6))=0$ and $h^{5}(\dE(c))=h^0(\dE(-c-6))=0$. Thus, $\chi(\dE(c))=h^0(\dE(c))-h^1(\dE(c))$. Therefore, $h^2(\dE nd\;\dE)=h^0(\dE(c))-\chi(\dE(c))$. Since $h^0(\dE(c))=h^0(\wedge^2\dH)-1$ (Corollary \ref{cor2.7}(iii)), then $h^2(\dE nd\;\dE)=h^0(\wedge^2\dH)-1-\chi(\dE(c))$. \\
(ii) Tensoring \eqref{seqD} by $\dE$, since $\dE$ is self-dual, we have the following short exact sequence
\begin{equation*}\label{5.1.4}
0\rightarrow\dE nd\;\dE\rightarrow\dE\otimes\dG\rightarrow\mathcal{E}(c)\rightarrow0
\end{equation*}
Passing to cohomology, we have the following exact sequence
\begin{equation*}
0\rightarrow H^0(\dE nd\;\dE)\rightarrow H^0(\dE\otimes\dG)\rightarrow H^0(\dE(c))
\end{equation*}
$$
\rightarrow H^1(\dE nd\;\dE)\rightarrow H^1(\dE\otimes\dG)\rightarrow H^1(\dE(c))
$$
\begin{equation*}
\rightarrow H^2(\dE nd\;\dE)\rightarrow H^2(\dE\otimes\dG)\rightarrow H^2(\dE(c))
\end{equation*}
where $H^2(\dE(c))=0$ by Proposition \ref{prop2.5} (ii). Thus,
\begin{equation}\label{RREQ}
-h^0(\dE nd\;\dE)+h^1(\dE nd\;\dE)-h^2(\dE nd\;\dE)
\end{equation}
\begin{equation*}
=-(h^0(\dE\otimes\dG)-h^1(\dE\otimes\dG)+h^2(\dE\otimes\dG))+(h^0(\dE(c))-h^1(\dE(c)))
\end{equation*}
Tensoring \eqref{seqA} by $\dG$, we have the following short exact sequence
$$
0\rightarrow \dG(-c)\rightarrow \dG\otimes\dF\rightarrow\dG\otimes\dE\rightarrow0
$$
Passing to cohomology, we have the following long exact sequence
$$
0\rightarrow H^0(\dG(-c))\rightarrow H^0(\dG\otimes\dF)\rightarrow H^0(\dG\otimes\dE)
$$
$$
\rightarrow H^1(\dG(-c))\rightarrow H^1(\dG\otimes\dF)\rightarrow H^1(\dG\otimes\dE)
$$
$$
\rightarrow H^2(\dG(-c))\rightarrow H^2(\dG\otimes\dF)\rightarrow H^2(\dG\otimes\dE)\rightarrow H^3(\dG(-c))
$$
where $h^i(\dG(-c))=0$ for $0\leq i\leq 3$ and $h^0(\dG\otimes\dF)=1$ since $\dG$ is stable and $\dG^*\cong\dF$. Thus,  
\begin{equation}\label{REQ6}
h^0(\dE\otimes\dG)=h^0(\dG\otimes\dF)=1
\end{equation}
and $h^i(\dG\otimes\dF)=h^i(\dG\otimes\dE), i=1,2$. Since $\dE$ is stable, then
\begin{equation}\label{REQ10}
h^0(\dE nd\;\dE)=1.
\end{equation}
Substituting (\ref{REQ6}) and (\ref{REQ10}) into (\ref{RREQ}), we have
\begin{equation}\label{RREQ1}
h^1(\dE nd\;\dE)-h^2(\dE nd\;\dE)=(h^1(\dE\otimes\dG)-h^2(\dE\otimes\dG))+(h^0(\dE(c))-h^1(\dE(c)))
\end{equation}
Tensoring \eqref{seqB} by $\mathcal{F}$, we have the following short exact sequence
$$
0\rightarrow\dF(-c)\rightarrow\dF\otimes\dH\rightarrow\dF\otimes\dG\rightarrow0
$$
Passing to cohomology, we have the following long exact sequence
$$
0\rightarrow H^0(\mathcal{F}(-c))\rightarrow H^0(\mathcal{F}\otimes\dH)\rightarrow H^0(\dF\otimes\dG)
$$
$$
\rightarrow H^1(\mathcal{F}(-c))\rightarrow H^1(\mathcal{F}\otimes\dH)\rightarrow H^1(\dF\otimes\dG)
$$
$$
\rightarrow H^2(\mathcal{F}(-c))\rightarrow H^2(\mathcal{F}\otimes\dH)\rightarrow H^2(\dF\otimes\dG)\rightarrow H^3(\mathcal{F}(-c))
$$
where $h^i(\mathcal{F}(-c))=0$ for $i=0,2,3$. Since $H^0(\mathcal{F}\otimes\dH)=0$ (by Corollary \ref{cor2.7}(i)) and $h^1(\mathcal{F}(-c))=h^0(\dF\otimes\dG)=1$, then $H^1(\mathcal{F}(-c))\cong H^0(\mathcal{F}\otimes\dG)$. Since $H^2(\mathcal{F}\otimes\dH)=0$, then $h^2(\dF\otimes\dG)=0$ and $H^1(\dF\otimes\dH)\cong H^1(\dF\otimes\dG)$.

Therefore, 
\begin{equation}\label{RREQ7}
h^2(\dE\otimes\dG)=h^2(\dG\otimes\dF)=0
\end{equation}
and $h^1(\dE\otimes \dG)=h^1(\dG\otimes\dF)=h^1(\dF\otimes\dH)=h^1(\dE\otimes\dH)$. Moreover, by Corollary \ref{cor2.7} (ii), 
\begin{equation}\label{RREQ8}
h^1(\dE\otimes \dG)=h^0(\dH(c))-h^0(\dH\otimes\dH)
\end{equation}
Substituting (\ref{RREQ7}) and (\ref{RREQ8}) into (\ref{RREQ}), we have
\begin{equation}\label{RREQ10}
h^1(\dE nd\;\dE)-h^2(\dE nd\;\dE)=h^0(\dE(c))-h^1(\dE(c))+h^0(\dH(c))-h^0(\dH\otimes\dH)
\end{equation}
By (i), $h^1(\dE nd\;\dE)=h^0(\wedge^2\dH)-1+h^0(\mathcal{H}(c))-h^0(\mathcal{H}\otimes \dH)$.
\end{proof}
\begin{remark}
By (\ref{seqA}) and (\ref{seqC}), one can see that $\chi(\dE(c))=h^0(\dH(c))-h^0(\mathcal{O}(2c))-1$. Therefore,
$h^2(\dE nd\;\dE)=h^0(\wedge^2\dH)+h^0(\mathcal{O}(2c))-h^0(\dH(c))$.
\end{remark}
Denote by $N(c,a_1,a_2,a_{3})$ the subvariety of $\mathcal{M}_{\mathbb{P}^{5}}(4;t_1,t_2,t_3,t_{4})$ parametrizing stable special generalized null correlation bundles on $\mathbb{P}^{5}$, where each $t_i$ is a polynomial in $c$, $a_1,a_2,a_{3}$ by Remark \ref{rmk2.21}. 
\begin{proposition}\label{thm4.4}
	$\dim N(c,a_1,a_2,a_{3})\geq h^0(\wedge^2\dH)-1+h^0(\dH(c))-h^0(\dH\otimes\dH)$.
\end{proposition}

\begin{proof}
Set \textit{}
$$
X:=\{\phi\in \Hom(\mathcal{O}(-c),\dH)|\phi\textit{\;is\;injective,}\,\coker\;\phi\;is\;locally\;free,\,\phi^TJ_{2n+2}\; is\;surjective\},
$$
where $J_{6}:=\begin{pmatrix}
0 & I_{3}\\
-I_{3} & 0
\end{pmatrix}$ and $I_{3}$ is the $3\times 3$ identity matrix. 
Then $X$ is a Zariski open subset of $\Hom(\mathcal{O}(-c),\dH)$. For any $\phi\in X$, we can define the following monad
$$
\mathcal{O}(-c)\xrightarrow{\phi}\dH\xrightarrow{\phi^TJ_{6}}\mathcal{O}(c)
$$
Denote by $\dE_\phi$ the cohomology bundle of this monad.
By Theorem \ref{thm3.5}, $\dE_\phi$ is stable for $\phi\in X$.

By the canonical projection  $pr_2:X\times\mathbb{P}^{5}\rightarrow\mathbb{P}^{5}$,
we can construct a universal family of the monad on $X\times\mathbb{P}^{5}$
$$
pr^*_2\mathcal{O}(-c)\rightarrow pr^*_2\mathcal{H}\rightarrow pr^*_2\mathcal{O}(c)
$$
This induces a morphism
$$
\tau:X\rightarrow\mathcal{M}_{\mathbb{P}^{5}}(4;t_1,t_2,t_3,t_{4})
$$
By definition $N(c,a_1,a_2,a_{3})=\tau(X)$. Since $X$ is irreducible, then $N(c,a_1,a_2,a_{3})$ is an irreducible subvariety of $\mathcal{M}_{\mathbb{P}^{5}}(4;t_1,t_2,t_3,t_{4})$.

Next we will give an explicit description of the fiber of the morphism $\tau$. Let $(\dE_1,e_1)$ and $(\dE_2,e_2)$ be two isomorphic special generalized null correlation bundles in $N(c,a_1,a_2,a_{3})$, where $e_i:\mathcal{E}_i\rightarrow\dE^*_i$ with $e^T_i=-e_i$. Since $\dE_1\cong\dE_2$, then there exists an isomorphism $\varphi:\dE_1\rightarrow\dE_2$. Furthermore, there exists a constant $\lambda\in\mathbb{C}^*$ such that $\lambda\varphi:\dE_1\rightarrow\dE_2$ is a symplectic isomorphism. In fact, since $\dE_1$ is self-dual and simple, $\varphi^Te_2\varphi\in \Hom(\dE_1,\dE^*_1)$ and $e_1\in\Hom(\dE_1,\dE^*_1)$, then $\varphi^Te_2\varphi=e_1/\lambda^2$ for some nonzero complex number $\lambda$. Thus, $\lambda\varphi$ is a symplectic isomorphism. Therefore, if $\dE_1$ is isomorphic to $\dE_2$, we can always assume that there exists a symplectic isomorphism $\varphi:\dE_1\rightarrow\dE_2$. Applying Theorem \ref{thm1.22} to $(\dE_1,e_1)$ and $(\dE_2,e_2)$, $\varphi$ is induced by the unique isomorphism of monads
$$
\xymatrix{
	{\mathcal{O}}(-c)\ar[r]^{a_1}\ar[d]^{\alpha} & \mathcal{H}\ar[r]^{a^T_1b_1}\ar[d]^{\beta}                         & \mathcal{O}(c)\ar[d]^{\alpha^{-1}}                         \\
	\mathcal{O}(-c)\ar[r]^{a_2} & \mathcal{H}\ar[r]^{a^T_2b_2}  & \mathcal{O}(c)  
}
$$
where $\alpha\in \Aut(\mathcal{O}(-c))$ and $\beta\in\Symp(\dH,J_{6})$.

On the other hand, for any $\phi\in X$, $G:=\Aut(\mathcal{O}(-c))\times\Symp(\dH,J_{6})$ operates on the space of self-dual monads via
$$
\phi\mapsto \beta \phi\alpha^{-1}, 
$$
where $(\alpha,\beta)\in G$. All monads in the same orbit define isomorphic special generalized null correlation bundles $\dE$. In conclusion, at least set theoretically $\tau$ induces a bijection from $X/G$ onto $N(c,a_1,a_2,a_{3})$. Then the proposition follows from Lemma \ref{lem2.11}.
\end{proof}
\begin{remark}
In fact, we can compute the dimension of fibers of $\tau$. For any $\phi\in G$, the isotropy group $G_\phi=\{\pm(\id_{\mathcal{O}(-c)},\,\id_{\dH})\}$ by Lemma 3.2 in \cite{BS92}. In particular, all isotropy groups $G_\phi$ are of dimension zero. Therefore $\tau$ has fiber of constant dimension $1+\dim\Symp(\dH,J_{6})$. Then we have
$
\dim N(c,a_1,a_2,a_{3})=h^0(\wedge^2\dH)-1+h^0(\dH(c))-h^0(\dH\otimes\dH).
$
\end{remark}

Similar to Theorem 3.1 in \cite{Ein88}, we have the following theorem.
\begin{theorem}\label{thm4.9} For $a_1\geq a_2\geq a_3\geq 0$ and $c>5a_1$,\\
	(i) the closure of $N(c,a_1,a_2,a_{3})$ in $\mathcal{M}_{\mathbb{P}^{5}}(4;t_1,t_2,t_3,t_{4})$ is an irreducible component of $\mathcal{M}_{\mathbb{P}^{5}}(4;t_1,t_2,t_3,t_{4})$;\\
	(ii) if $p\in N(c,a_1,a_2,a_{3})$, then $p$ is a smooth point of $\mathcal{M}_{\mathbb{P}^{5}}(4;t_1,t_2,t_3,t_{4})$.
\end{theorem}
\begin{proof}
(i)\&(ii)	Let $\dE_p$ be the special generalized null correlation bundle corresponding to $p$. Since $H^1(\dE nd\;\dE_p)$ is the Zariski tangent space at of $\mathcal{M}_{\mathbb{P}^{5}}(4;t_1,t_2,t_3,t_{4})$ at $p$, then $h^1(\dE nd\;\dE_p)\geq \dim N(c,a_1,a_2,a_{3})$. By Proposition \ref{thm4.1} and \ref{thm4.4}, $h^1(\dE nd\;\dE_p)=\dim N(c,a_1,a_2,a_{3})$. Then the theorem follows immediately. 
\end{proof}

\begin{remark}
Recall from \cite{KO89,LO87,N79}, we know that the coarse moduli space of simple vector bundles (in analytic sense) on complex projective space exists as a (not necessarily Hausdorff) analytic space. And for any $p$ in the moduli space, the dimension of the Zariski tangent space of the moduli space at $p$ is equal to $h^1(\dE nd\;\dE_p)$, where $\dE_p$ is the simple bundle corresponding to $p$. Since we only use the simplicity of $\dE$, $\dF$ and $\dG$ in the proofs of Proposition \ref{thm4.1} and \ref{thm4.4}, then one can establish similar results for moduli spaces of simple special generalized null correlation bundles by the same argument. 	 
\end{remark}

\section{Application}

In this section, we assume that $\dE$ is a special generalized null correlation bundle on $\mathbb{P}^5$ with $c>5a_1$. Then $\dE$ is stable.
\begin{lemma}\label{lem5.5}
 $c_1(\dE)=c_3(\dE)=0$, $c_2(\dE)=c^2-(a^2_1+a^2_2+a^2_3)$, $c_4(\dE)=c^4-c^2(a^2_1+a^2_2+a^2_3)+(a^2_1a^2_2+a^2_1a^2_3+a^2_2a^2_3)$.	
\end{lemma}
\begin{proof}
The lemma follows from \eqref{seqB} and \eqref{seqD} easily.
\end{proof}
\begin{lemma}\label{lem5.8}
	For any positive integer $N$, there exists a pair of positive integers $(\lambda,\zeta)$ such that the following system of equations  
	\[
	\begin{cases}
	\lambda= u^2+v^2+w^2\\
	\zeta= u^2\cdot v^2+u^2\cdot w^2+v^2\cdot w^2
	\end{cases}
	\]
	has at least N distinct integral solutions satisfying $u\geq v\geq w\geq0$.
\end{lemma}

\begin{proof}
	Since
	$$
	u^2\cdot v^2+u^2\cdot w^2+v^2\cdot w^2=\frac{1}{2}((u^2+v^2+w^2)^2-(u^4+v^4+w^4))
	$$
	it is enough to show the following claim is true. 
	\begin{claim}\label{claim5.9}
		For any positive integer $N$, there exist a pair of positive integers $(A,B)$ such that the following system of equations 
		\[
		\begin{cases}
		4A= u^2+v^2+w^2\\
		4B= u^4+v^4+w^4
		\end{cases}
		\]
		has at least N distinct integral solutions satisfying $u\geq v\geq w\geq0$.
	\end{claim}
	If Claim \ref{claim5.9} is true, we can let $\lambda=4A$ and $\zeta=((4A)^2-4B)/2=8A^2-2B$. Then the lemma follows. For the proof of Claim \ref{claim5.9}, we need to use the following identities in \cite{Piezas}.
	For $k=2,4$,
	$$
	(ax+(a+2b)y)^k+(bx-(2a+b)y)^k+((a+b)x-(a-b)y)^k=(a^k+b^k+(a+b)^k)(x^2+3y^2)^{k/2}.
	$$
	For general choices of $(a,b)$,
	$
	\begin{pmatrix}
	a && a+2b\\
	b && -2a-b 
	\end{pmatrix}
	$
	is nondegenerate. Now we fix such $(a,b)$. Since $x^2+3y^2=1$ has infinitely many rational solutions, then for any positive integer $N$, there exists a positive integer $M$ such that $x^2+3y^2=M$ has at least $N$ distinct integral solutions. Now we can just let $A$ be $(a^2+b^2+(a+b)^2)M$ and $B$ be $4(a^4+b^4+(a+b)^4)M^2$. Then Claim \ref{claim5.9} follows. The proof of the lemma is complete.
\end{proof}
\begin{proposition}\label{prop6.8}
	Suppose that $\mathcal{M}_{\mathbb{P}^5}(4;0,s,0,t)=X_1\cup X_2\cdots\cup X_{m_{s,t}}$ where $X_i$'s are distinct irreducible components of $\mathcal{M}_{\mathbb{P}^5}(4;0,s,0,t)$. Then there exists a sequence of pairs $\{(s_i,t_i)\in\mathbb{Z}^2|i=1,2,\cdots\}$ such that $\lim_{i\rightarrow+\infty} m_{s_i,t_i}=+\infty$. 
\end{proposition}
\begin{proof}
	The proposition follows from Theorem \ref{thm4.9}, Lemma \ref{lem5.5} (i) and Lemma \ref{lem5.8} immediately.
\end{proof}
\begin{remark}
By Bohnhorst and Spindler's theorems on the smoothness and irreducibility of moduli spaces of stable homological dimension 1 vector bundles on projective spaces (see Theorem 3.5 in \cite{BS92}), one can show that the number of irreducible components of moduli space of stable rank 3 (resp. 5) bundles on $\mathbb{P}^3$ (resp. $\mathbb{P}^5$) with fixed Chern classes can be arbitrarily high. In fact, one only needs to consider the moduli space of stable vector bundles with some fixed Chern classes which contains the isomorphism class of $\dG$ on $\mathbb{P}^3$ and $\mathbb{P}^5$ respectively, where $\dG$ is as defined in the display of monad (\ref{eqn13}). 
\end{remark}
\section*{Acknowledgement}

The author is greatly indebted to his advisor, Lawrence Ein, for suggesting the problem and many stimulating conversations. The author wishes to express his thanks to Klaus Hulek and Aroor P. Rao for answering his questions on symplectic vector bundles. The author would also like to thank Izzet Coskun, Ben Gould, Fumiaki Suzuki and Tian Wang for helpful discussions and suggestions.
\bibliographystyle{abbrv}
\bibliography{reference}

\end{document}